\documentclass[11pt, reqno,amscd, psamsfonts]{amsart}
\usepackage{amsthm, amsmath}
\usepackage{amssymb} 
\usepackage{amscd}
\usepackage{graphicx}
\usepackage{amsfonts}
\usepackage{mathrsfs}
\usepackage{color}
\usepackage{soul}
\usepackage{hyperref}
\newtheorem{theorem}[subsection]{Theorem}
\newtheorem{definition}[subsection]{Definition}
\newtheorem{example}[subsection]{Example}
\newtheorem{question}[subsection]{Question}
\newtheorem{cor}[subsection]{Corollary}

\newtheorem{lemma}[subsection]{Lemma}
\newtheorem{proposition}[subsection]{Proposition}

\newtheorem{sublemma}[subsubsection]{Lemma}

\newtheorem{subproposition}[subsubsection]{Proposition}

\def\X{{X}}

\def\Tc{{\mathscr{T}}}

\def\Xc{{\mathscr{X}}}
\def\Uc{{\mathscr{U}}}
\def\Yc{{\mathscr{Y}}}
\def\Zc{{\mathcal{Z}}}
\def\Hc{{\mathscr{H}}}

\def\arrowdown#1#2{\Big\downarrow \rlap{$\vcenter{\hbox{$\scriptstyle#2$}}$}
{\hbox to -10pt{\hss{$\vcenter{\hbox{$\scriptstyle#1$}}$}}}}
\def\arrowup#1#2{\Big\uparrow \rlap{$\vcenter{\hbox{$\scriptstyle#2$}}$}
{\hbox to -10pt{\hss{$\vcenter{\hbox{$\scriptstyle#1$}}$}}}}

\numberwithin{equation}{section}

\begin{document}
\title [Finite morphisms to projective space and capacity theory]{Finite morphisms to projective space and capacity theory}

\date{\today}

\begin{abstract}We study conditions on a commutative ring $R$ which are equivalent to the 
following requirement:  Whenever $\Xc$ is a projective scheme over $S = \mathrm{Spec}(R)$ of fiber dimension $\le d$
for some integer $d \ge 0$, there is a finite morphism from $\Xc$ to $\mathbb{P}^d_S$ over $S$ such
that the pullbacks of coordinate hyperplanes give prescribed subschemes of $\Xc$ provided these subschemes
satisfy certain natural conditions. We use our results to define a new kind of capacity
for adelic subsets of  projective flat schemes $\X$ 
over global fields.  This capacity can be  used to generalize the converse part of
the Fekete-Szeg\H{o} Theorem. 

\end{abstract}

\author[T. Chinburg]{T. Chinburg}\thanks{Chinburg is supported in part by NSF Grant DMS-1100355}
\address{Ted Chinburg, Dept. of Math\\Univ. of Penn.\\Phila. PA. 19104, U.S.A.}
\email{ted@math.upenn.edu}

\author[L. Moret-Bailly]{L. Moret-Bailly}\thanks{Moret-Bailly is supported in part by the ANR project ``Points entiers et points rationnels''}
\address{Laurent Moret-Bailly, IRMAR, Universit\'e de Rennes 1 \\Campus de Beaulieu,
F-35042 Rennes Cedex, France
}
\email{Laurent.Moret-Bailly@univ-rennes1.fr}

\author[G. Pappas]{G. Pappas}
\thanks{Pappas is supported in part by NSF Grant DMS11-02208}
\address{Georgios Pappas, Dept. of Math\\Michigan State Univ.\\East Lansing, MI 48824, U.S.A.  }
\email{pappas@math.msu.edu}

\author[M. Taylor]{M. J. Taylor}
\address{Martin J. Taylor, School of Mathematics\\ Merton College, Univ. of Oxford\\ Merton Street, Oxford OX1 4JD, U.K.}
\email{martin.taylor@merton.ox.ac.uk}

\date{\today}

\maketitle



\section{Introduction}
\label{s:intro}

Let $R$ be a commutative ring.  Suppose $f:\Xc \to S = \mathrm{Spec}(R)$ 
is a projective morphism which has fiber dimension $\le d$ in the sense
that  every fiber of $f$ has dimension $\le d$ at each of its points.
(If $d < 0$ this just means that $\Xc$ is empty.)   
Fix a line bundle $\mathscr{L}$
on $\Xc$ which is ample relative to $f$.  Suppose $i$ is an integer in
the range $0 \le i \le d+1$.  Let $(h_1,\ldots,h_i)$ be a sequence of
sections of $\mathscr{L}$ such that when $\Xc_j$
is the zero locus of $h_j$, then $\cap_{j = 1}^i \Xc_j$ has fiber dimension
$\le d -i$.  

\begin{definition}
\label{def:coordhyp}  The ring $R$ has the coordinate hyperplane property {\rm (F)} 
if for all $\Xc$, $\mathscr{L}$, $i$, $\{h_j\}_{j = 1}^i$ and $\{\Xc_j\}_{j = 1}^i$ as above, the following is true.
There is a finite $S$-morphism
$\pi:\Xc \to \mathbb{P}^d_S$ and a set of homogenous coordinates $(y_1:\cdots:y_{d+1})$ for $\mathbb{P}^d_S$ 
such that for each $1 \le j \le i$, the support of $\Xc_j$ is 
the same as that of the subscheme of $\Xc$ defined by $\pi^* y_j  = 0$.
\end{definition}

Recall that $\mathrm{Pic}(R)$ is the group of isomorphism classes of invertible sheaves on $S$, or equivalently, the group
of isomorphism classes of rank one projective $R$-modules under tensor product.  Our first Theorem is the following result:

\begin{theorem}
\label{thm:main}
The following properties of a commutative  ring $R$ are equivalent:
\begin{enumerate}
\item[1.] {\rm (Property F)} The coordinate hyperplane property.
\item[2.]  {\rm (Property P)} For every finite $R$-algebra $R'$, $\mathrm{Pic}(R')$ is a torsion
group.
\item[3.] {\rm (Property S)} If $\Uc$ is an open subscheme of $\mathbb{P}^1_S$ which surjects onto
$S = \mathrm{Spec}(R)$, there is an $S$-morphism $\Yc \to \Uc$ for which $\Yc$ is finite,
locally free and surjective over $S$.
\end{enumerate}
\end{theorem}

Property (P) holds, for example, if $R$ is semi-local or the  localization of an order inside
a ring of integers of a global field.  Property (P) is stable under taking finite algebras and quotients.
It is stable  under filtering direct limits of rings because  projective $R$-modules
of rank one are of finite presentation (as direct summands of free finitely generated modules) - see   \cite[Prop. 1.3]{Vas}.  A zero dimensional ring satisfies (P).  A finitely generated
algebra $R$ over a field $k$ satisfies (P) if and only if $k$ is algebraic over a finite field
and $\mathrm{dim}(R) \le 1$. 

Suppose $R$ has the coordinate hyperplane property.  The proof of Theorem \ref{thm:main} shows
that with the notations of Definition \ref{def:coordhyp}, one can find a $\pi:\Xc \to \mathbb{P}^d_S$ such
that $\pi^* O_{\mathbb{P}^d_S}(1)$ is a power of $\mathcal{L}$ and the subscheme of $\Xc$ defined
by $\pi^* y_j$ is an effective Cartier divisor which is a positive integral multiple of $\Xc_j$ for $1 \le j \le i$. (See \S \ref{s:SgivesF}).  If
$i = 0$  in Theorem \ref{def:coordhyp} one obtains that there is a finite 
$S$-morphism
$\pi:\Xc \to \mathbb{P}^d_S$.  When $R$ is the ring of integers of a number field
and $d = 1$, the existence of such a $\pi$ was shown by Green in \cite{Green} and
\cite{Green2}.  Green's result is used in \cite{CPT} to reduce the proof of
certain adelic Riemann-Roch Theorems on surfaces to the case of $\mathbb{P}^1_S$.  We have learned that for
arbitrary $d$, the existence of
a $\pi$ as above when $R$ is a Dedekind ring satisfying Property P has
been shown by different methods by O. Gabber, Q. Liu and D. Lorenzini
in \cite{Gabber}.  

In \S \ref{s:equivalences} we will give some equivalent formulations of properties (P) and (S).
The proof of Theorem \ref{thm:main} is given in \S \ref{s:proof}.  We do not know if the conclusion
of property (F)  when $i = 0$ is sufficient to imply property (F) holds, i.e. whether
this implies the same conclusion for arbitrary $i$. All of properties (P), (S) and (F) make sense over 
a scheme $S$ which may not be affine.  It would be interesting to consider the relationship
of these properties for such a scheme $S$.

We now discuss an application of Theorem \ref{thm:main} to capacity theory.

Suppose $\X$ is a projective flat connected normal
scheme of dimension $d\ge 1$ over a global field $F$.   Let $\X_1$ be an effective ample Cartier divisor on $\X$. Let $M(F)$ be the set of places of $F$.  Let $\mathbb{C}_v$ be the completion of an algebraic closure $\overline{F}_v$ of the completion $F_v$ of $F$ at $v \in M(F)$.
By an adelic set $\mathbb{E} = \prod_{v \in M(F)} E_v$ we will mean a product of subsets of $E_v$ of $(\X \smallsetminus  \X_1)(\mathbb{C}_v)$
which satisfy the following standard hypotheses relative to $\X_1$:  

\label{eq:conditions}
\begin{enumerate}
\item[i.] Each $E_v$ is nonempty and stable under the group $\mathrm{Gal}^c(\mathbb{C}_v/K_v)$
of continuous automorphisms of $\mathbb{C}_v$ over $K_v$.
\item[ii.] Each $E_v$ is bounded away from $\mathrm{supp}(\X_1)(\mathbb{C}_v)$ under the
$v$-adic metric induced by the given projective embedding of $\X$, and for all but finitely 
many $v$, $E_v$ and $\mathrm{supp}(\X_1)(\mathbb{C}_v)$ reduce to disjoint sets (mod $v$).
\end{enumerate}

Let $S_\gamma(\mathbb{E},\X_1)$ be the sectional capacity of $\mathbb{E}$ relative
to $\X_1$ (c.f. \cite{C} and \cite{RLV}).    Let $\overline {F}$ be an algebraic closure of $F$.  The main arithmetic interest of $S_\gamma(\mathbb{E},\X_1)$ is that if this number is less than $1$, then there is a open adelic neighborhood $\mathbb{U} = \prod_{v \in M(F)} U_v$ of $\mathbb{E}$ such that the set of points $\X(\overline{F})$ which have all of their Galois conjugates in $\mathbb{U}$
is not Zariski dense. 

In  \S \ref{s:finitemor} we  define a real number $\gamma_{\mathrm{F}}(\mathbb{E},\X_1)$ which will be called the finite morphism capacity of $\mathbb{E}$ relative to $\X_1$. 
This number gives a higher dimensional
converse implication, paralleling the classical  Fekete-Szeg\H{o} Theorem and the work of Cantor and Rumely
discussed below when $d= 1$:

\begin{theorem}
\label{thm:converseFS}  If $\gamma_{\mathrm{F}}(\mathbb{E},\X_1)  > 1$ then every open adelic neighborhood $\mathbb{U}$
of $\mathbb{E}$ contains a Zariski dense set of points of $\X(\overline{F})$ which have all of their Galois
conjugates in $\mathbb{U}$.
\end{theorem}

The definition of $\gamma_{\mathrm{F}}(\mathbb{E},\X_1)$ involves considering
the set $T(\X_1)$ of all finite morphims $\pi:\X \to \mathbb{P}^d_{F}$ such that the pull back via $\pi$ of a hyperplane $H$ at infinity
is a positive integral multiple of $\X_1$. Let $ \mathcal{U}(\mathbb{E})$ be the set of open adelic neighborhoods $\mathbb{U}$ of $\mathbb{E}$ which satisfy the standard hypotheses.  Roughly speaking,  $\gamma_{\mathrm{F}}(\mathbb{E},\X_1)$ is an infimum over $\mathbb{U} \in \mathcal{U}(\mathbb{E})$ of  the supremum of the normalized size of the  adelic polydiscs $B$  in affine $d$-space $A^d_{F} = \mathbb{P}^d_{F} \smallsetminus  H$
for which $\pi^{-1}(B)$ is contained in $\mathbb{U}$ for some $\pi \in T(\X_1)$.  

Our first result about $\gamma_{\mathrm{F}}(\mathbb{E},\X_1)$ 
compares it to the 
outer sectional capacity $S^+_\gamma(\mathbb{E},\X_1)  =  \inf \{ S_\gamma(\mathbb{U},\X_1): {\mathbb{U} \in \mathcal{U}(\mathbb{E})}\}$.

\begin{theorem}
\label{thm:compare} One has $$\gamma_{\mathrm{F}}(\mathbb{E},\X_1) \le S_\gamma^+(\mathbb{E},\X_1)$$  
with equality if $\mathbb{E} = \pi^{-1}(B)$ for some
$\pi \in T(\X_1)$ and some adelic polydisk $B$ in $\mathbb{P}^d$.  
\end{theorem}

It will not in general be the case that $\gamma_{\mathrm{F}}(\mathbb{E},\X_1) = S^+_\gamma(\mathbb{E},\X_1)$.
For example, there may be effective Cartier divisors $\X_1'$  with the same support as $\X_1$ such that
 $$S_\gamma^+(\mathbb{E},\X'_1) < 1 < S_\gamma^+(\mathbb{E},\X_1).$$  The first inequality
implies the set of points  $\X(\overline{F})$ which have all their Galois conjugates in some
open adelic neighborhood $\mathbb{U}$ of $\mathbb{E}$ is not Zariski dense.   
So Theorem \ref{thm:converseFS} implies $\gamma_{\mathrm{F}}(\mathbb{E},\X_1) \le 1 <  S_\gamma^+(\mathbb{E},\X_1)$.

This suggests that in comparing sectional capacity to finite morphism capacity, one should consider all effective $\X'_1$ with the same support as $\X_1$.   
One would like to minimize sectional capacity and maximize
finite morphism capacity.  One must also reckon with the fact that if $m > 0 $ an integer, replacing $\X'_1$ by $m \X'_1$ raises
each of $S^+_\gamma(\mathbb{E},\X'_1)$ and $\gamma_{\mathrm{F}}(\mathbb{E},\X'_1)$ to the $m^{d+1}$-th power. This leads to the following definition:

\begin{definition}
\label{def:adjust}  Suppose $\X$ is smooth over $F$.  Let $D(\X_1)$ be the countable set of all ample effective Cartier divisors $\X'_1$ 
on $\X$ whose support $\mathrm{supp}(\X'_1)$ equals $\mathrm{supp}(\X_1)$. Let $|X'_1| > 0$ be the
$d$-fold self intersection number of $\X'_1$.  \begin{enumerate}
\item[i.] Let $S^+_\gamma(\mathbb{E},\mathrm{supp}(\X_1))$ be  
$$ \inf \{ S_\gamma(\mathbb{U},\X'_1)^{|X'_1|^{-(d+1)/d}}: {\mathbb{U} \in \mathcal{U}(\mathbb{E}),\ \  \X'_1 \in D(\X_1)} \}$$
\item[ii.]  Let $\gamma_{\mathrm{F}}(\mathbb{E},\mathrm{supp}(\X_1))$ be  
$$ \inf_{\mathbb{U} \in \mathcal{U}(\mathbb{E})} \left \{ \sup_{\X'_1 \in D(\X_1)} \gamma_{\mathrm{F}}(\mathbb{U},\X'_1)^{|X'_1|^{-(d+1)/d}}\right \}.$$
\end{enumerate}
\end{definition}


\begin{question}
\label{eq:qup} 
If $S^+_\gamma(\mathbb{E},\mathrm{supp}(\X_1)) > 1$ is  $\gamma_{\mathrm{F}}(\mathbb{E},\mathrm{supp}(\X_1)) > 1$?
\end{question}

Note that if Question \ref{eq:qup} has an affirmative answer for a given $\X$, this and Theorem \ref{thm:compare} gives a complete 
Fekete-Szeg\H{o} Theorem for $\X$ and the capacity $S^+_\gamma(\mathbb{E},\mathrm{supp}(\X_1))$.  Namely, one would know that if $S^+_\gamma(\mathbb{E},\mathrm{supp}(\X_1)) < 1$, then there is a open adelic neighborhood $\mathbb{U}$ of $\mathbb{E}$ such that the set of points $\X(\overline{F})$ which have all of their Galois conjugates in $\mathbb{U}$
is not Zariski dense, while if $S^+_\gamma(\mathbb{E},\mathrm{supp}(\X_1)) > 1$ then every such $\mathbb{U}$ contains a Zariski dense set
of such conjugates. 

For $\X$ of dimension $1$ Rumely has compared sectional capacity with his generalization of Cantor's
capacity, and he has proved a complete Fekete-Szeg\H{o} Theorem of the above kind.  We will use his
work to check that when $d = 1$,  Question \ref{eq:qup} has an affirmative answer in a very strong sense:

\begin{theorem}
\label{eq:curve} Suppose $d = 1$ and that $\mathbb{E}$ is $\mathrm{supp}(\X)$-capacitable in the sense of \cite[\S 6.2]{Rumely}.  
If $S^+_\gamma(\mathbb{E},\mathrm{supp}(\X_1)) > 1$ then  $\gamma_{\mathrm{F}}(\mathbb{E},\mathrm{supp}(\X_1))$ and\break  
$S_\gamma^+(\mathbb{E},\mathrm{supp}(\X_1))$ are both equal to  
the capacity $\gamma_{\mathrm{CR}}(\mathbb{E},\mathrm{supp}(\X_1))$ of Cantor and Rumely.
\end{theorem}

It would be interesting to know if such an equality holds for $\X$ of higher dimension when
$S^+_\gamma(\mathbb{E},\mathrm{supp}(\X_1)) > 1$.  A natural case to consider is when $\X$ itself is
isomorphic to $\mathbb{P}^d$. Then each map $\phi:\X = \mathbb{P}^d \to \mathbb{P}^d$
in the definition of $\gamma_{\mathrm{F}}(\mathbb{E},\X_1)$ together with the line bundle
$O_{\X}(\X_1)$ gives rise to a polarized dynamical system in the sense of \cite[\S 2]{P}.
In a later paper we will return to the problem of how to compute the sectional and finite
morphism capacities of filled Julia sets associated to dynamical systems on $\mathbb{P}^d$.

Simple examples show that one cannot expect the conclusion of Theorem \ref{eq:curve} to hold
if $S^+_\gamma(\mathbb{E},\mathrm{supp}(\X_1)) < 1$ (see Example \ref{ex:counter}). 
These examples suggest that in this case, there may simply not be enough morphisms from $\X$
to $\mathbb{P}^d$ to explain the numerical size of $S^+_\gamma(\mathbb{E},\mathrm{supp}(\X_1))$.

We conclude with an application of Theorem \ref{thm:main} to finite morphism capacity.
If $F$ is a number field, let  $M_{\mathrm{f}}(F)$ and $M_\infty(F)$ be the sets of finite and infinite places of $F$.  If $F$ is a global function field, we  let $M_\infty(F)$
be a nonempty finite set of places of $F$, and we let $M_{\mathrm{f}}(F) = M(F) \smallsetminus  M_\infty(F)$. For $v \in M_{\mathrm{f}}(F)$ let 
$O_{\mathbb{C}_v} = \{z \in \mathbb{C}_v:  |z|_v \le 1\}$. 

 We will say that a product  
$\mathbb{U}_{\mathrm{f}} = \prod_{v \in M_{\mathrm{f}}(F)}  U_v$ of subsets $U_v \subset \X(\mathbb{C}_v)$  is an RL-domain for the finite adeles if 
there is a projective embedding 
$\X \to \mathbb{P}^t_F$  and a hyperplane $H_0$ in $\mathbb{P}^t_F$ such that $\X_1$ is a multiple of $ \X \cap H_0$ and 
$$U_v = \X(\mathbb{C}_v) \cap (\mathbb{P}^t \smallsetminus  H_0)(O_{\mathbb{C}_v})$$  for all $v \in M_{\mathrm{f}}(F)$.  It is clear that an adelic set $\mathbb{U} = \mathbb{U}_{\mathrm{f}} \times \prod_{v \in M_\infty(F)} U_v$
will satisfy the standard hypotheses provided the $U_v$ associated to $v \in M_\infty(F)$  satisfy conditions (i) and  (ii) above. 

Note that the dimension $d$ of $\X$ may be much smaller than that of the projective space $\mathbb{P}_F^t$ into
which we have embedded $\X$.  Recall that in $\mathbb{P}_F^d$ we have chosen homogeneous
coordinates $y = (y_1:\ldots:y_{d+1})$ as well as affine coordinates $z_i = y_i/y_1$ for $i = 2,\ldots,d+1$
for the affine space $A^d = \mathbb{P}^d \smallsetminus  H$ when $H = \{y: y_1 = 0\}$.  We will show:

\begin{theorem}
\label{thm:urk}
Suppose that $\mathbb{U}_{\mathrm{f}}$ is an RL-domain with respect to the finite adeles. There is a finite morphism $\pi_{\mathbb{U}_{\mathrm{f}}}:\X \to \mathbb{P}_F^d$ with the following properties:
\begin{enumerate}
\item[a.] The pullback $\pi_{\mathbb{U}_{\mathrm{f}}}^{-1}(H)$ is a positive integral multiple of $\X_1$.
\item[b.]  Let  $B_{\mathrm{f}}(1) = \prod_{v \in M_{\mathrm{f}}(F)} B_v(1)$ where 
$B_v(1)$ is the unit polydisk $$\{z = (z_2,\ldots,z_{d+1}): |z_i|_v \le 1\} \subset A^d(\mathbb{C}_v).$$
Then 
$$\pi^{-1}_{\mathbb{U}_{\mathrm{f}}}(B_{\mathrm{f}}(1)) = \mathbb{U}_{\mathrm{f}}.$$
\end{enumerate}
\end{theorem}

\begin{cor}
\label{cor:app}Given $\mathbb{U}_{\mathrm{f}}$ as in Theorem \ref{thm:urk}, there is an open adelic set 
$\mathbb{U} = \mathbb{U}_{\mathrm{f}} \times \prod_{v \in M_\infty(F)} U_v$ which satisfies the standard hypotheses
such that $\gamma_{\mathrm{F}}(\mathbb{U},\X_1) > 1$.  In consequence
the set of global algebraic points of $\X(\overline{F})$ which have all of their Galois conjugates in $\mathbb{U}_{\mathrm{f}}$
is Zariski dense.  
\end{cor}

We prove Theorem \ref{thm:urk} by applying Theorem \ref{thm:main} to a well chosen integral model of
$\X$ over the natural choice of Dedekind subring $O_F$ of $F$.  

To 
obtain more quantitative information about finite morphism capacities and to study Question \ref{eq:qup}, one needs a version of Theorem \ref{thm:main} which includes conditions at $v \in M_\infty(F)$.  This amounts to a question in function theory, in
the sense that one must construct rational functions on $\X$ giving sections
of a power of $O_{\X}(\X_1)$ which define a morphism to $\mathbb{P}^d$
having the required properties.

\section{Some properties of rings}
\label{s:equivalences}

\begin{proposition}
\label{prop:Picprop}
The following conditions on a commutative ring $R$ are equivalent:
\begin{enumerate}
\item[1.] For every $R$-algebra $R'$ which is integral over $R$, $\mathrm{Pic}(R')$ is
a torsion group.
\item[2.] For every finite $R$-algebra $R'$, $\mathrm{Pic}(R')$ is a torsion group.
\item[3.] For every finite and finitely presented $R$-algebra $R'$, $\mathrm{Pic}(R')$
is a torsion group.
\end{enumerate}
We shall say that $R$ satisfies property {\rm (P)} if these conditions hold.
\end{proposition}

\begin{proof}It is obvious that (1) implies (2) and (2) implies (3).  To show that (3) implies (1),
note that every integral $R'$-algebra is a filtering direct limit of finite and finitely
presented $R$-algebras.  The result then follows from the fact that the Picard
functor commutes with filtering direct limits, as noted in the introduction.  
\end{proof}

The second condition on $R$ we will consider is the ``Skolem property" (S)
(see \cite{MB}).

\begin{proposition}
\label{prop:Skolprop}For a commutative ring $R$ with spectrum $S$, the following conditions
are equivalent:
\begin{enumerate}
\item[1.] For each $n \in \mathbb{N}$ and each open subscheme $\Uc \subset \mathbb{P}^n_S$
which is surjective over $S$, there is a subscheme $\Yc$ of $\Uc$ which is finite, free and surjective
over $S$.
\item[2.]  Same as condition \textup{(1)}, with $n = 1$.
\item[3.] For each $n \in \mathbb{N}$ and each open subscheme $\Uc \subset \mathbb{P}^n_S$
which is surjective over $S$, there is 
an $S$-morphism $\Yc\to \Uc$ where $\Yc$ is finite, locally free and surjective
over $S$.
\item[4.]  Same as condition \textup{(3)}, with $n = 1$.
\end{enumerate}
We will say that $R$ has property {\rm (S)} if these conditions hold.
\end{proposition}

To begin the proof of Proposition \ref{prop:Skolprop}, we first note that we can equivalently use $\mathbb{A}^n_S$
instead of $\mathbb{P}^n_S$ in each case. It is trivial that (1) implies both (2) and (3), and that
either (2) or (3) implies (4).  The fact that (2) implies (1), and that (4) implies (3), is shown by the following result:

\begin{sublemma}
\label{lem:trick}  Let $\Uc \subset \mathbb{A}_S^n$ be open and surjective over $S$.  Then there exists
an $n$-tuple of positive integers $m_1 = 1,m_2,\ldots,m_n$ such that if $j$ is the closed
immersion $j:\mathbb{A}^1_S \to \mathbb{A}^n_S$ defined by $j(t) = (t,t^{m_2},\ldots,t^{m_n})$
then the open subset $j^{-1}(\Uc)$ is surjective over $S$.
\end{sublemma}
\begin{proof}
We may assume that $\Uc$ is quasi-compact.  The  complement of $\Uc$ in $\mathbb{A}^n_S$ is then
defined by a finite set of polynomials $f_i \in R[X_1,\ldots,X_n]$.  The surjectivity of $\Uc$
means that the coefficients of the $f_i$'s generate the unit ideal of $R$.  Consider the finite
set $\Sigma \subset \mathbb{N}^n$ of all multi-exponents occurring in the $f_i$'s.  It is easy
to see that one can find positive integers $m_2,\ldots,m_n$ such that the linear
form $(x_1,\ldots,x_n) \to x_1 + \sum_{\ell = 2}^n m_\ell x_\ell$ maps $\Sigma$ injectively
into $\mathbb{N}$.  But this means that for each $i$ the polynomial $f_i(t, t^{m_2},\ldots,t^{m_n}) \in R[t]$
has the same set of coefficients as $f_i$.  In particular, these coefficients still generate $R$.
\end{proof}

To complete the proof of Proposition \ref{prop:Skolprop}, it will suffice  to show that (4) implies (2).
Let $\Uc \subset \mathbb{A}_S^1$ be open and surjective over $S$.  Choose an $S$-morphism
$\Yc \to \Uc$ as in (4).  On passing to  a locally free cover of $\Yc$, we may assume that $\Yc = \mathrm{Spec}(R_1)$ is locally free of 
constant (positive) rank $r$ over $S$.  The composite map $\Yc \to \mathbb{A}^1_R  = \mathrm{Spec}(R[t])$
gives rise to a morphism $R[t] \to R_1$ mapping $t$ to an element $z$. Let $F(t) \in R[t]$ denote
the characteristic polynomial of $z$, and put $\Yc' = \mathrm{Spec}(R[t]/(F(t)))$.  Then $\Yc'$ is finite
and free of rank $r$ over $R$, and it is easy to check that $\Yc'$ is set theoretically the
image of $\Yc$ because $\Yc \to \Uc$ factors through $\Yc'$ by the Cayley-Hamilton Theorem.
In particular, $\Yc' \subset \Uc$ as desired.\qed

\section{Proof of Theorem \ref{thm:main}}
\label{s:proof}

\subsection{Property (P) implies property (S)}

As usual, put $S = \mathrm{Spec}(R)$, and assume $R$ has property (P).  Consider
an open subscheme $\Uc \subset \mathbb{P}^1_S$ which is surjective over $S$.  It will 
suffice to show that there is a subscheme $\Yc \subset \Uc$ which is finite, locally free and
surjective over $S$.  We may assume that $\Uc$ is quasicompact and contained in $\mathbb{A}^1_S$.
Let $\Zc \subset \mathbb{P}^1_S$ be a closed subscheme of finite presentation 
with support $\mathbb{P}^1_S \smallsetminus  \Uc$.  Then $\Zc$ is finite over $S$.  By property (P),
there is an integer $m > 0$ such that the invertible $\mathscr{O}_\Zc$-module $\mathscr{O}_\Zc(m)$
is trivial, i.e. has a nonvanishing section $s_\Zc$.  Now $\Zc$ and $s_\Zc$ can be defined over a subring $R_0 \subset R$
which is finitely generated over $\mathbb{Z}$.  To find $\Yc$ we may replace $R$ by $R_0$, so
that we may now assume $R$ is Noetherian.   The  ideal sheaf $\mathscr{I}_\Zc$ is coherent, so replacing $m$ by a sufficiently
large multiple, we may assume that $\mathrm{H}^1(\mathscr{I}_\Zc(m)) = 0$. This  implies that
the restriction map $\mathrm{H}^0(\mathscr{O}_{\mathbb{P}^1}(m)) \to 
\mathrm{H}^0(\mathscr{O}_{\Zc}(m)) $ is surjective.  In particular, $s_\Zc$ extends to a section
$s \in \mathrm{H}^0(\mathscr{O}_{\mathbb{P}^1}(m))$.  We can take $\Yc$ to be the scheme of
zeros of $s$.  Indeed, the fact that $s_\Zc$ is a trivialization on $\Zc$ means that $\Yc \subset \Uc \subset  \mathbb{A}^1_S$.
We can therefore view $s$ as a polynomial of degree $m$ in the standard coordinate $t$
of $\mathbb{A}^1_S$.  The leading coefficient of this polynomial is invertible because $s$
does not vanish at infinity, so $\Yc \cong \mathrm{Spec}(R[t]/(s))$ is free of rank $m$ over $S$.

\subsection{Property (S) implies property (F)}
\label{s:SgivesF}

We will need the following general fact.

\begin{subproposition}
\label{prop:allright}
 Without assumption on the commutative ring $R$, let $g:\Yc \to S = \mathrm{Spec}(R)$
be a projective morphism with fiber dimension $\leq\delta$ for some integer $\delta\geq0$.  Let $\mathscr{L}$
be an invertible sheaf on $\Yc$ which is very ample with respect to $g$.  After
base change from $S$ to an $S$-scheme $S'$ which is surjective over $S$ and isomorphic
to an open subscheme of an an affine $S$-space $\mathbb{A}^N_S$, there is a section
of $\mathscr{L}$ over $\Yc$ whose scheme of zeros has fiber dimension $\le \delta-1$.
If, moreover, $R$ satisfies condition {\rm (S)}, there is an integer $m > 0$ and a section
of $\mathscr{L}^{\otimes m}$ on $\Yc$ (without any base change) with the same
property.
\end{subproposition}

\begin{proof} We may assume that $\Yc \subset \mathbb{P}^{N-1}_S$ and $\mathscr{L} = \mathscr{O}_\Yc(1)$
since $S$ is affine.  We may identify sections of $\mathscr{O}_{\mathbb{P}^{N-1}_S}(1)$ with sections of the
vector bundle $E := \mathbb{A}^N_S$ over $S$.  This identification is compatible with base change.  In particular,
we have a universal section of $\mathscr{O}_{\mathbb{P}^{N-1}_E}(1)$ whose scheme of zeros in $\Yc$
is the universal hyperplane section $\Hc \subset \Yc \times_S E \subset \mathbb{P}^{N-1}_S \times E = \mathbb{P}^{N-1}_E$.
By Chevalley's semi-continuity theorem (\cite[13.1.5]{GD}), the locus $S' \subset E$ over which the fibers of $\Hc$
have dimension $\leq\delta-1$ is open.  Moreover, $S'$ surjects onto $S$ since for each geometric point $\xi:\mathrm{Spec}(k) \to S$
of $S$ there is a hyperplane in $\mathbb{P}^{N-1}_k$ which does not contain any component of the fiber $\Yc_\xi$.  Thus $S'$
provides the required base change.

Assume now that $R$ satisfies condition (S).  Applying this property to the above $S'$, we obtain a finite,
free and surjective $S$-scheme $\pi:T \to S$ contained in $S'$.  By restricting sections to $T$,
we obtain a section $h$ of the pullback of $\mathscr{L}$ to $\Yc \times_S T$ whose zero set $\Zc$ has
fiber dimension $\leq\delta-1$ over $T$.  Denote by $m > 0$ the degree of $\pi$. The natural projection
$\pi_\Yc:\Yc \times_S T \to \Yc$ is still free of degree $m$.  The norm of $h$ with respect to $\pi_\Yc$ is
a section of $\mathscr{L}^{\otimes m}$ on $\Yc$.  The zero set of this section is $\pi_\Yc(\Zc)$, which has
the same fiber dimension as $\Zc$ since $\pi_\Yc$ is finite.  This completes the proof.
\end{proof}

We may now show that property (S) implies that $R$ has the coordinate hyperplane property (F).
Let $d$, $f:\Xc \to S$, $\mathscr{L}$, $i$ and $h_1,\ldots,h_i$ be as in Definition \ref{def:coordhyp}.
By replacing $\mathscr{L}$ by $\mathscr{L}^{\otimes e}$ for some large enough $e$, and each
$h_j$ by $h_j^{\otimes e}$, we may assume that $\mathscr{L}$ is very ample.  We apply Proposition 
\ref{prop:allright} inductively, starting with the $S$-scheme $\Yc = \cap_{j =1}^i \Xc_j$ and the integer $\delta=d-i$.  We get (after replacing
$\mathscr{L}$ and the $h_j$'s by suitable powers of themselves, which does not change the $\Xc_j$'s) 
sections $h_1,\ldots,h_{d+1}$  of $\mathscr{L}$ whose common zero set has fiber dimension
$\le -1$, i.e.\ is empty.  Therefore $(h_1:\cdots:h_{d+1})$ is a well defined $S$-morphism
$q:\Xc \to \mathbb{P}^d_S$.  Moreover, $q^* \mathscr{O}_{\mathbb{P}^d}(1) = \mathscr{L}$ is ample, so
$q$ must be finite.  By construction, for all $j \le i$, the pullback of the $j^{th}$ homogeneous coordinate
is a power of $h_j$, hence its zero set is set theoretically equal to $\Xc_j$ and equal as a Cartier divisor to a multiple
of $\Xc_j$.  This completes the proof that property
(S) implies property (F).

\subsection{Property (F) implies property (P)}
\label{s:CtoP}

We will need the following general result.

\begin{subproposition}
\label{prop:flatprop}  Let $S$ be a scheme, and suppose $\Xc$ and $\Yc$ are two finitely presented
$S$-schemes.  Let $\pi:\Xc \to \Yc$ be a  finite $S$-morphism. Assume that $\Yc \to S$ is
flat, with pure $d$-dimensional regular fibers.  Then $\pi$ is flat if and only if $\Xc \to S$
is flat with pure $d$-dimensional Cohen-Macaulay fibers.
\end{subproposition}

\begin{proof} Combine  {\cite[Thm. 46, p. 140]{Mat}}, applied to the fibers, with ``flatness by fibers" \cite[IV, 3, 11.3.11]{GD}.\end{proof}

In order to now prove that a ring $R$ has property (P), we will in fact only use property (F) in the special case 
$d = i = 1$ in the notation of Definition \ref{def:coordhyp}.  Let $R'$ be a finite $R$-algebra and define
$S' = \mathrm{Spec}(R')$.  Suppose $M$ is an invertible $R'$-module.  We must show that $M^{\otimes m} \cong R'$
for some $m > 0$.  Consider the locally free rank one $\mathscr{O}_{S'}$-module $\mathscr{M}$ associated to
$M$, and the corresponding $\mathbb{P}^1$ bundle $\Xc = \mathbb{P}(\mathscr{O}_{S'} \oplus \mathscr{M})$.
This bundle has two disjoint natural sections over $S'$:  The section $s_\infty$ whose complement is isomorphic to
the vector bundle $\mathbb{M} = \mathbb{V}(\mathscr{M})$ and the zero section $s_0$ of $\mathbb{M}$.
These sections define divisors $D_\infty$ and $D_0$ which are ample with respect to $S'$ (and therefore
also ample with respect to $S$).  Put $\mathscr{L} = \mathscr{O}_\Xc(D_\infty + D_0)$, and let $h$ be the canonical
section of $\mathscr{L}$ having divisor $\Yc = D_\infty + D_0$.  We can now apply property (F) to this
data.  We obtain a finite morphism $\Xc \to \mathbb{P}^1_S$ such that $\Yc$ is the set-theoretic inverse image of, say,
the section $\infty$ of $\mathbb{P}^1_S$.  Since $\Xc$ is an $S'$-scheme, this gives rise to a finite $S'$-morphism
$p:\Xc \to \mathbb{P}^1_{S'}$ mapping $\Yc$ to the section $\infty$.  By Proposition \ref{prop:flatprop},
$p$ must be flat since $\Xc$ and $\mathbb{P}^1_{S'}$ are smooth and one-dimensional over $S'$.
We conclude that $p$ is in fact locally free since it is finite, flat and of finite presentation.  Clearly $p$
is also surjective.  The inverse image of the zero section of $\mathbb{P}^1_{S'}$ is therefore a finite
locally free $S'$-scheme $T$ which surjects onto $S'$ and which is contained in the punctured line
bundle $\Xc\smallsetminus \Yc = \mathbb{M}\smallsetminus D_0$.  This means that $\mathscr{M}$ is trivialized by
the base change $T \to S'$, so $\mathscr{M}$ has finite order in $\mathrm{Pic}(S')$.

\section{Finite morphism capacities}
\label{s:finitemor}

\subsection{The definition}
\label{s:defgencap}
Let $F$ be a global field, and let $M(F)$ be the set of places of $F$.
Define $\mathbb{C}_v$ to be the completion of an algebraic closure $\overline{F}_v$
of the completion $F_v$ of $F$ at $v$.  Define $| \ |_v:\mathbb{C}_v \to \mathbb{R}$
to be the unique extension to $\mathbb{C}_v$ of the normalized absolute value
on $F_v$.    Define $\mathcal{R}$ to be the set of functions $r:M(F) \to \mathbb{R}$
such that $r(v) \ge 0 $ for all $v$ and $r(v) = 1$ for almost all $v$.  We have a norm
$| \ |:\mathcal{R} \to \mathbb{R}$ defined by $$|r| = \prod_{v \in M(F)} r(v).$$

 Let   $d \ge 1$ be an integer.
 We let $(y_1:\cdots:y_{d+1})$ be homogeneous coordinates on $\mathbb{P}^d_{F}$,
 and we define $H$ to be the hyperplane $y_1 = 0$. Then we have affine
 coordinates on $A^d = \mathbb{P}^d \smallsetminus H$ given by $z_i = y_i/y_1$ for $i = 2,\ldots,d+1$.  
 For  $r \in \mathcal{R}$ we define the adelic polydisk $B(r) = \prod_{v \in M(F)} B_v(r(v))$
 by setting 
 $$B_v(r(v)) = \{z = (z_2,\ldots,z_{d+1}) \in A^d(\mathbb{C}_v): |z_i |_v \le r(v) \quad \mathrm{for\ all }\quad i\}$$

  Let $\X_1$ be an ample effective divisor on the smooth projective variety $\X$ over $F$.
Suppose $\mathbb{E} = \prod_{v \in M(F)} E_v$ is an adelic set of points of $\X$ satisfying the
standard hypotheses described in (i) and (ii) of the introduction.  Recall that $\mathcal{U}(\mathbb{E})$
is the set of all open adelic neighborhoods $\mathbb{U} = \prod_v U_v$
 of $\mathbb{E}$ which satisfy the standard hypotheses.   
Let $T(\X_1)$ be the set of all finite morphisms $\pi:X \to \mathbb{P}^d$ over $F$ such that $\pi^{-1} (H) = m(\pi) \X_1$
as Cartier divisors for some integer $m(\pi) > 0$.  

\begin{definition}
\label{def:finim}The finite morphism capacity $\gamma_{\mathrm{F}}(\mathbb{U},\X_1)$ of $\mathbb{U} \in \mathcal{U}(\mathbb{E})$
relative to $\X_1$ is  the supremum of 
$$|r|^{d \cdot \mathrm{deg}(\pi)/m(\pi)^{d+1}}$$
over all $\pi \in T(\X_1)$ and $r \in \mathcal{R}$
such that $\pi^{-1}(B(r)) \subset \mathbb{U}$;
this supremum is defined to be $0$ if no such $r$ and $\pi$ exist.  Define 
\begin{equation}
\label{eq:gin}
\gamma_{\mathrm{F}}(\mathbb{E},\X_1) = \inf \{ \gamma_{\mathrm{F}}(\mathbb{U},\X_1): \mathbb{U} \in \mathcal{U}(\mathbb{E})\}.
\end{equation}
\end{definition}

\subsection{Generalizing the converse part of the Fekete-Szeg\H{o} Theorem.} 
\label{s:conversepart}

In this paragraph we will prove Theorem \ref{thm:converseFS}.  Suppose \hbox{$\gamma_{\mathrm{F}}(\mathbb{E},\X_1) > 1.$}
We must prove that for every  $\mathbb{U} \in \mathcal{U}(\mathbb{E})$,
the set of points of $\X(\overline{F})$ which have all their Galois conjugates in $\mathbb{U}$ is Zariski dense. 
In view of (\ref{eq:gin}), it will suffice to consider the case in which $\mathbb{E} = \mathbb{U}$ is open.  

Since $\gamma_{\mathrm{F}}(\mathbb{U},\X_1) >1$, there are $\pi$ and $r$ as in Definition \ref{def:finim} for which $|r| > 1$.   Since $\pi:\X \to \mathbb{P}_F^d$ is finite and defined over $F$, the result now follows from:

\begin{lemma}
\label{lem:projZardense}  If $r \in \mathcal{R}$ and $|r| > 1$ then there is a Zariski dense
set of points of $\mathbb{P}^d(\overline{F})$ which have all their Galois conjugates over $F$ in $B(r)$.
\end{lemma}

\begin{proof}  By multiplying $r$ by a function from $M(F)$ to $\mathbb{R}$ of the form $v \to |\alpha|_v^{1/n}$ for a suitable $\alpha \in F^*$
and positive integer $n$, we can reduce to the case in which  $r(v) \ge 1$ for all $v$ and $r(v) > 1$ if $v \in M_\infty(F)$.    Let $C$ be the set of points 
$(z_2,\ldots,z_{d+1}) \in A^d(\overline{F})$ such that each $z_i$ is a root of unity.  Every point of $C$ has all its conjugates in $B(r)$, so it will suffice to show $C$
is Zariski dense in $\mathbb{P}^d_F$, or equivalently in $A^d(\overline{F})=\overline{F}^d$.  This follows from the well-known fact that if $I$ is an infinite subset of a field $K$, then 
$I^d$ is Zariski dense in $K^d$.
\end{proof}
 
 \subsection{Comparing sectional capacity and finite morphism capacity.}

 \begin{lemma}
\label{lem:projectivespace} Suppose $\pi \in T(\X_1)$ and $r \in \mathcal{R}$.  The sectional capacity $S_\gamma(\pi^{-1}(B(r)),\X_1)$ equals $|r|^{d \cdot \mathrm{deg}(\pi)/m(\pi)^{d+1}}$.
\end{lemma}

\begin{proof}   By the functorial properties of sectional capacity proved in Theorem C of the introduction of
 \cite{RLV} we have
\begin{align}
 \label{eq:gradual}
 S_\gamma(\pi^{-1}(B(r)),\X_1) &=& S_\gamma(\pi^{-1}(B(r)),\pi^{-1}(\X_1))^{1/m(\pi)^{d+1}}\nonumber\\
 &=&  S_\gamma(B(r),H)^{\mathrm{deg}(\pi)/m(\pi)^{d+1}}.
 \end{align}
 We are thus reduced to showing that on $\mathbb{P}^d_F$ one has $S_\gamma(B(r),H) = |r|^d$.  
This follows immediately from Example 4.3 of \cite{LR} and an explicit (classical) computation
when $d = 1$ (see Examples 4.1 and 4.2 of \cite{LR}).
\end{proof} 

\noindent {\bf Proof of Theorem \ref{thm:compare}.}

To show $\gamma_{\mathrm{F}}(\mathbb{E},\X_1) \le S_\gamma^+(\mathbb{E},\X_1)$
it will suffice to consider the case in which $\mathbb{E}$ is an open adelic set $\mathbb{U}$.  
In view of Lemma \ref{lem:projectivespace}, $\gamma_{\mathrm{F}}(\mathbb{U},\X_1)$ is the supremum of
$$|r|^{d \cdot \mathrm{deg}(\pi)/m(\pi)^{d+1}} = S_\gamma(\pi^{-1}(B(r)),\X_1)$$
over $\pi \in T(\X_1)$ 
and $r \in \mathcal{R}$ such that $\pi^{-1}(B(r)) \subset \mathbb{U}$.  Since 
$\pi^{-1}(B(r)) \subset \mathbb{U} $ we have
$$S_\gamma(\pi^{-1}(B(r)),\X_1) \le 
S_\gamma(\mathbb{U},\X_1) = S_\gamma^+(\mathbb{U},\X_1).$$ Thus 
$\gamma_{\mathrm{F}}(\mathbb{U},\X_1)$  is the supremum of a set of numbers all of
which are bounded by $S_\gamma(\mathbb{U},\X_1)$. Therefore 
$\gamma_{\mathrm{F}}(\mathbb{U},\X_1) \le S_\gamma(\mathbb{U},\X_1)$.

Suppose now that $\mathbb{E} = \pi^{-1}(B(r))$ for some $\pi$ and $r \in \mathcal{R}$ above.
Then $\pi^{-1}(B(r)) = \mathbb{E}$ is trivially contained in every adelic open neighborhood $\mathbb{U}$
of $\mathbb{E}$. So the above calculations show  
$\gamma_{\mathrm{F}}(\pi^{-1}(B(r)),\X_1) = S_\gamma(\pi^{-1}(B(r)),\X_1)$.

\subsection{The case of curves.}
\label{s:curves}

In this section we prove Theorem \ref{eq:curve}, whose notations we now assume.
Thus $\X$ is a regular connected projective curve over $F$.  We may
write the effective divisor $\X_1$ on $\X$ as a positive integral combination
$\sum_i s_i x_i$ of a finite $\mathrm{Gal}(\overline{F}/F)$-stable set of points $\Tc = \{x_i\}_{i =1}^n$
of $\X(\overline{F})$.   Let $\mathbb{E}$ be an adelic set which satisfies
the standard hypotheses in the introduction and which is $\Tc$-capacitable in the sense of \cite[\S 6.2]{Rumely}.  Rumely has proved in \cite{R2} that the sectional capacity
$S_\gamma(\mathbb{E},\X_1)$ satisfies
\begin{equation}
\label{eq:Rumform}
-\ln(S_\gamma(\mathbb{E},\X_1)) = s^t \Gamma(\mathbb{E},\Tc) s
\end{equation}
when $s$ is the column vector $s = (s_i)_{i = 1}^n$
and $\Gamma(\mathbb{E},\Tc)$ is the real symmetric $n \times n$ Green's matrix arising in the
capacity theory of Cantor and Rumely. 

Recall that $D(\X_1)$ is the set of all ample effective Cartier divisors $\X'_1$ 
on $\X$ whose support $\mathrm{supp}(\X'_1)$ equals $\Tc = \mathrm{supp}(\X_1)$. In
the case of curves, $|X'_1|$ is the degree of $X'_1$, and  
$$S^+_\gamma(\mathbb{E},\Tc) =  \inf \{ S_\gamma(\mathbb{U},\X'_1)^{|X'_1|^{-2}}: {\mathbb{U} \in \mathcal{U}(\mathbb{E}),\ \  \X'_1 \in D(\X_1)} \}.$$

A real vector $r' = (r'_i)_{i = 1}^n$ will be called $F$-symmetric if $r_i = r_j$ when $x_i$ and $x_j$ are elements of $\Tc \subset \X(\overline{F})$ 
in the same orbit under the action of $\mathrm{Aut}(\overline{F}/F)$.
Let $\mathcal{P}$ be the set of $F$-symmetric vectors $r' = (r'_i)_{i = 1}^n$ such that $0 \le r'_i \in \mathbb{R}$ for all $i$ and 
$\sum_{i = 1}^n r'_i = 1$.   Let $\mathcal{P}^0$ be the elements of $\mathcal{P}$ which have
only positive entries. In \cite{R2}, Rumely has shown that sectional capacities are well defined for formal
positive real linear integral combinations $\X'_1 = \sum_i r'_i x_i$ of the points $x_i \in \Tc$
such that $r = (r_i)_i \in \mathcal{P}^0$.  

\begin{lemma}
\label{lem:firststep} Suppose $S^+_\gamma(\mathbb{E},\Tc) > 1$.  Then $\Gamma(\mathbb{E},\Tc)$ is a negative definite matrix. There
is a unique $\hat{s} = (\hat{s}_i)_i \in \mathcal{P}^0$ such that 
$$\hat{s}^t \Gamma(\mathbb{E},\Tc) \hat{s} = 
-\ln(S^+_\gamma(\mathbb{E},\Tc)).$$
 If $X'_1 =  \sum_i \hat{s}_{i}  \ x_i$ then 
\begin{equation}
\label{eq:easy}
S^+_\gamma(\mathbb{E},\Tc) = 
S_\gamma(\mathbb{E},\X'_1) = \gamma_{\mathrm{CR}}(\mathbb{E},\Tc).
\end{equation}
where $\gamma_{\mathrm{CR}}(\mathbb{E},\Tc)$ is the Cantor Rumely
capacity.
\end{lemma}

\begin{proof} Suppose $\mathbb{U} \in \mathcal{U} = \mathcal{U}(\mathbb{E})$.  It follows from (\ref{eq:Rumform}) that 
\begin{equation}
\label{eq:lneq} 
\sup \{ \frac{-\ln (S_\gamma(\mathbb{U},\X'_1))}{|X'_1|^2}:  \X'_1 \in D(\X_1) \}
 = \mathrm{sup}_{s \in \mathcal{P}^0}  \ \ s^t \Gamma(\mathbb{U},\Tc) s.
\end{equation}
Thus
\begin{equation}
\label{eq:lneq2}
- \ln(S^+_\gamma(\mathbb{E},\Tc)) =  
\sup_{\mathbb{U} \in \mathcal{U}} \sup_{s \in \mathcal{P}^0}  \ \ s^t \Gamma(\mathbb{U},\Tc) s.
\end{equation}
By the approximation theorems of \cite[\S 4.5]{Rumely}, for every $\epsilon > 0$,
there is a $\mathbb{U} \in \mathcal{U}(\mathbb{E})$ such that every entry of
$\Gamma(\mathbb{U},\Tc) - \Gamma(\mathbb{E},\Tc)$
is bounded by $\epsilon$ in absolute value.

Let $\mathrm{Val}(\Gamma)$ be the value of a square $n \times n$ matrix $\Gamma$ as a matrix game (see \cite[p. 327]{Rumely}).  Then 
\begin{equation}
\label{eq:define}
\gamma_{\mathrm{CR}}(\mathbb{E},\Tc) = \mathrm{exp}(-\mathrm{Val}(\Gamma(\mathbb{E},\Tc)))
\end{equation} by definition (c.f. \cite[Def. 5.1.5]{Rumely}). 

Suppose
first that $\mathrm{Val}(\Gamma(\mathbb{E},\Tc)) \ge 0$. Let $\delta > 0$ be given. We can then find a $\mathbb{U} \in \mathcal{U}$ such that $\mathrm{Val}(\Gamma(\mathbb{U},\Tc)) \ge -\delta$.
Let $J$ be the matrix
of the same size as $\Gamma(\mathbb{U},\Tc)$ which has every entry equal to $1$.  Then $\Gamma(\mathbb{U},\Tc)+ 2 \delta J$
is a symmetric matrix whose only negative entries are on the diagonal,
and $$\mathrm{Val}(\Gamma(\mathbb{U},\Tc)+ 2 \delta J) = \mathrm{Val}(\Gamma(\mathbb{U},\Tc)) + 2 \delta   > 0.$$ Hence by
\cite[Lemma 5.1.7]{Rumely}, there is a vector $s_0 \in \mathcal{P}^0$ such
that $$(\Gamma(\mathbb{U},\Tc)+ 2 \delta J) s_0$$ has all positive entries. Then
$0 < s_0^t \Gamma(\mathbb{U},\Tc) s_0 + s_0^t 2 \delta J s_0  = s_0^t \Gamma(\mathbb{U},\Tc) s_0 +2 \delta$, so
$$\sup_{s \in \mathcal{P}^0}  \ \ s^t \Gamma(\mathbb{U},\Tc) s \ge -2 \delta.$$
Since $\delta > 0$ was arbitrary,
we conclude from (\ref{eq:lneq2}) that $- \ln(S^+_\gamma(\mathbb{E},\Tc)) \ge 0$. This contradicts  the assumption that $S^+_\gamma(\mathbb{E},\Tc) > 1$. Hence $\mathrm{Val}(\Gamma(\mathbb{E},\Tc)) < 0$.  

We know from \cite[Thm. 5.1.10]{Rumely} that 
\begin{equation}
\label{eq:lowerbound}
\mathrm{Val}(\Gamma(\mathbb{E},\Tc)) \ge \mathrm{Val}(\Gamma(\mathbb{U},\Tc))
\end{equation} 
for $\mathbb{U} \in \mathcal{U}$.  Thus if $\mathbb{V} \in \mathcal{U}$ or $\mathbb{V} = \mathbb{E}$ and $\Gamma = \Gamma(\mathbb{V},\Tc)$, we have
$\mathrm{Val}(\Gamma) < 0$. 
Rumely shows 
in \cite[Lemma 5.1.7]{Rumely} that this implies $ s \to s^t \Gamma s$ is a negative definite quadratic form. Hence this quadratic form
has a  maximum on the space $W$ of all real $F$-symmetric vectors $s = (s_i)_i$  
such that $\sum_i s_i = 1$ .  The matrix $\Gamma = (\tau_{ij})_{i,j}$ is symmetric and invertible, and
$\tau_{i', j'} = \tau_{i,j}$ if $x_{i'} = \sigma(x_i)$ and $x_{j'} = \sigma(x_j)$ for some $\sigma \in \mathrm{Aut}(\overline{F}/F)$. Calculus
implies that at each  $\hat{s} \in W$ where the maximum of $ s \to s^t \Gamma s$ is obtained, the 
vector $\Gamma \hat{s}$ must have all equal and nonzero components. If $s'$ is
another such vector, then $s' - q\hat{s}$ is in the kernel of $\Gamma$ for some
scalar $q$.  Because $\Gamma$ is invertible, we conclude that $s' = q \hat{s}$ so in fact $s' = \hat{s}$ since  $s', \hat{s} \in W$.
Hence $\hat{s}$ is unique.  Rumely shows  in \cite[Lemma 5.1.6]{Rumely} that the maximum of
$ s \to s^t \Gamma s$ is attained at a point $\hat{s} \in \mathcal{P}^0$,
where all the components of $\Gamma \hat{s}$ are equal to $\mathrm{Val}(\Gamma)$.
Thus the quadratic form $s \to s^t \Gamma s$ achieves its (unique) maximum over
the compact set $\mathcal{P}$ at the interior point $\hat{s} \in \mathcal{P}^0$,
and this maximum is $\mathrm{Val}(\Gamma)$.

In view of (\ref{eq:lowerbound}), and the fact that we can find a sequence of $\mathbb{U} \in \mathcal{U}(\mathbb{E})$ for which the corresponding matrices 
$\Gamma(\mathbb{U},\Tc)$ converge to $\Gamma(\mathbb{E},\Tc)$,
we conclude from (\ref{eq:lneq2}) that
$$- \ln(S^+_\gamma(\mathbb{E},\Tc)) = \mathrm{Val}(\Gamma(\mathbb{E},\Tc)).$$
The Lemma now follows from (\ref{eq:define}) together with the uniqueness
noted above for the $s \in \mathcal{P}$ at which $s \to s^t \Gamma(\mathbb{E},\Tc) s$
achieves its maximum.
\end{proof}

\begin{lemma}
\label{lem:longhair}  Suppose that $\mathbb{U} \in \mathcal{U} = \mathcal{U}(\mathbb{E})$.  Let $W$ be the set of all pairs $(f,r)$ of the following kind.  The
first entry of $(f,r)$ is a finite morphism 
$f:X \to \mathbb{P}^1$ such that $\hat{\X}_1 = f^{-1}(\infty)$ has the same support $\Tc$ as $\X_1$. The second entry 
of $(f,r)$ is an element  $r \in \mathcal{R}$
such that 
$$f^{-1}(B(r)) \subset \mathbb{U}.$$
Then 
\begin{equation}
\label{eq:finally}
\sup_{\X_1' \in D(\X_1)} \gamma_{\mathrm{F}}(\mathbb{U},\X'_1)^{|\X'_1|^{-2}} = \sup_{(f,r) \in W} 
\gamma_{CR}(f^{-1}(B(r)),\Tc) \le \gamma_{CR}(\mathbb{U},\Tc)
\end{equation}
where $D(\X_1)$ is the set of all ample effective divisors having support
$\Tc$.
\end{lemma}

\begin{proof} Since each $\mathbb{U} \in \mathcal{U}$ is open, we have for $\X'_1 \in D(\X_1)$ that 
$$\gamma_{\mathrm{F}}(\mathbb{U},\X'_1) = \sup\{|r|^{\mathrm{deg}(f)/m(f)^2}: (f,r) \in W \quad \mathrm{and}\quad f^{-1}(\infty) = m(f) \X'_1\}$$
On the right hand side we have $m(f) |\X'_1| = \mathrm{deg}(f)$.   So by the pullback formula for the Cantor-Rumely capacity (c.f. \cite[p. 4]{Rumely}) we have 
$$(|r|^{\mathrm{deg}(f)/m(f)^2})^{|\X'_1|^{-2}} = |r|^{1/\mathrm{deg}(f)}  = \gamma_{CR}(B(r),\infty)^{1/\mathrm{deg}(f)} = \gamma_{CR}(f^{-1}(B(r)),\Tc).$$
Combining the above equalities shows (\ref{eq:finally}), where 
$\gamma_{CR}(f^{-1}(B(r)),\Tc) \le \gamma_{CR}(\mathbb{U},\Tc)$ because $f^{-1}(B(r)) \subset \mathbb{U}$.
\end{proof}

\begin{lemma}
\label{lem:construct}  With the hypotheses and notations of Lemma
\ref{lem:firststep}, suppose $\mathbb{U} \in \mathcal{U}(\mathbb{E})$  and  $\epsilon > 0$.  Then there is a finite morphism 
$f:\X \to \mathbb{P}^1$ with the following properties.  Let $n(f)$ be the degree of
$f$.  
\begin{enumerate}
\item[i.] The pull back divisor $f^{-1}(\infty)$ has support $\Tc$.
\item[ii.]    
There is an adelic polydisk $B(r)$ in $\mathbb{A}^1$ for
some $r \in \mathcal{R}$ 
such that $f^{-1}(B(r)) \subset \mathbb{U}$ and
$|r|^{1/n(f)} = \gamma_{CR}(f^{-1}(B(r)),\Tc)$
is  within $\epsilon$
of $ \gamma_{\mathrm{CR}}(\mathbb{E},\Tc)$.
\end{enumerate}
In consequence
\begin{equation}
\label{eq:finally2}
\sup_{\X_1' \in D(\X_1)} \gamma_{\mathrm{F}}(\mathbb{U},\X'_1)^{|\X'_1|^{-2}} \ge \gamma_{CR}(\mathbb{E},\Tc)
\end{equation}
\end{lemma}

\begin{proof}  By Lemma \ref{lem:firststep}, our hypotheses imply
that $\gamma_{\mathrm{CR}}(\mathbb{E},\Tc) > 1$. 
The result stated by Rumely in \cite[Theorem 6.2.2]{Rumely}
is that there is a finite morphism $f:\X \to \mathbb{P}^1$ for which (i) holds
and such that $f^{-1}(B(r_0)) \subset \mathbb{U}$ when $r_0 \in \mathcal{R}$
has $r_0(v) = 1$ for all $v \in M(F)$.  By reading the proof  closely,
one sees that Rumely shows the sharper quantitative result in  (ii).  We will
make a few comments about how to check this before leaving the details
to the reader.  

The preliminary reductions started in Step 0 on page 395 
of \cite{Rumely} are not completed until Step 7 on page 412. To check
that the reduction step  which involves a  base change to a finite normal extension 
$L$ of $K$ still applies, one needs
to use the behavior of the Cantor-Rumely capacity under such base changes
which is stated as (B) on page 4 of \cite{Rumely}.   In Step 1 on page 396, 
constants $R_w$ are defined for each place $w$ in $L$ in the union $S = S_0 \cup \{w_0\}$
of a large finite set of place $S_0$ with a nonarchimedean place $w_0$ 
where the $w_0$-component $E_{w_0}$ of $\mathbb{E}$ is $\Tc$-trivial.
Define $R_w = 1$ for all $w \in M(L) \smallsetminus  S$ and let $\Gamma = \Gamma(\mathbb{E}_L,\Tc)$
be the Green's matrix associated to the pullback of $\mathbb{E}$ to $L$.    
Define $R = \prod_{w \in M(L)} R_w$.  

In the function field case, Rumely
arranges by the end of Case A of Step 1 on page 399 of \cite{Rumely} that
$\ln(R)/n =  - \mathrm{Val}(\Gamma)$ when $n$ is the common degree of
the functions $h_w(z)$ constructed in this step.  In the number field case,
he shows by the end of Case B of Step 1 on 403 that one can arrange for
$\ln(R)/n$ to be as close as one likes to $- \mathrm{Val}(\Gamma)$.  Steps
2, 3 and 4 of the proof proceed as stated.  In case A of step 5 on page 411
of \cite{Rumely}, Rumely shows the lower  bound
$$|f_L(z)|_w  \ge \frac{1}{2} |h(z)^d|_w \ge \frac{1}{2} R_w^d  > 1$$
for $z \in \mathcal{Y}_w$ with the notations introduced there.
Here $f_L:L \otimes_K \X \to \mathbb{P}^1_L$
is the morphism we need to construct when $\X$ is replaced by $L \otimes_K \X$
via the first reduction step.  
Rumely uses only the lower bound $|f_L(z)|_w > 1$ at this step
in the proof.  One has to carry along the sharper bound
$|f_L(z)|_w  \ge \frac{1}{2} R_w^d$ and let $d$ go to infinity,
where the degree of $f_L(z)$ is $d$ times the degree of the function
$h(z)$ appearing in equation (79) on page 411.  Since $\ln(R )/n= \sum_{w \in M(L)} \ln(R_w)/n$
can be made as close as we like to $- \mathrm{Val}(\Gamma)$, we find that
$f_L$ satisfies the counterparts for $\X_L$ of conditions (i) and (ii) because
  $ \gamma_{\mathrm{CR}}(\mathbb{E}_L,\Tc) = 
\mathrm{exp}(-\mathrm{Val}(\Gamma))$.  We get an $f$ as in (i) and (ii)
via the reduction step.  Letting $f^{-1}(\infty) = \X'_1$, we see that 
(\ref{eq:finally2}) follows from the first equality in (\ref{eq:finally}) and from 
that fact that $\epsilon$ was an arbitrary positive number in condition (ii)
of Lemma \ref{lem:construct}.  
 \end{proof}

\noindent {\bf Completion of the proof of Theorem  \ref{eq:curve}.}
\medbreak
In view of Lemma \ref{lem:firststep}, it will suffice to show that \begin{equation}
\label{eq:formula}
\gamma^+_F(\mathbb{E},\Tc) = \inf_{\mathbb{U} \in \mathcal{U}} \left \{ \sup_{\X'_1 \in D(\X_1)} \gamma_{\mathrm{F}}(\mathbb{U},\X'_1)^{|X'_1|^{-2}}\right \}.
\end{equation}
equals $\gamma_{\mathrm{CR}}(\mathbb{E},\Tc)$.  
Lemma \ref{lem:longhair} implies that (\ref{eq:formula})
is bounded above by 
$$\inf_{\mathbb{U} \in \mathcal{U}} \gamma_{\mathrm{CR}}(\mathbb{U},\Tc)
$$
and this equals $\gamma_{\mathrm{CR}}(\mathbb{E},\Tc)$ by the approximation
Theorems of \cite[\S 4.5]{Rumely}.  Thus $\gamma_{\mathrm{F}}^+(\mathbb{E},\Tc)  \le \gamma_{\mathrm{CR}}(\mathbb{E},\Tc)$. Lemma \ref{lem:construct}
and (\ref{eq:formula}) give $\gamma_{\mathrm{F}}^+(\mathbb{E},\Tc)  \ge \gamma_{\mathrm{CR}}(\mathbb{E},\Tc)$, which completes the proof.
 
 \begin{example}
 \label{ex:counter} {\rm We give an example showing that if $S_\gamma^+(\mathbb{E},\Tc)  < 1$ then one may
 have $\gamma_{\mathrm{F}}^+(\mathbb{E},\Tc) = 0$.  Suppose that $F = \mathbb{Q}$, $\X = \mathbb{P}^1$ and $\X_1$ is the point $\infty$.  Define $\mathbb{E} = \prod_{v \in M(\mathbb{Q})} E_v$ by letting $E_v = \mathbb{A}^1(O_{\mathbb{C}_v})$ if $v$ is finite, and by letting $E_v$ the open disc of radius $1/2$ about
 $1/2$ in $\mathbb{A}^1(\mathbb{C}_{v_\infty}) = \mathbb{C}$ when $v_\infty$ is the archimedean place.  Then
 $S_\gamma(\mathbb{E},\infty) = S_\gamma^+(\mathbb{E},\infty)= 1/2$.    Suppose  $f:\X \to \mathbb{P}^1$ is a finite morphism such that  $\mathrm{supp}(f^{-1}(\infty)) = \{ \infty\}$ and
 $f^{-1}(B(r)) \subset \mathbb{E}$ for some $r \in \mathbb{R}$ with $|r| > 0$.  Then $f^{-1}(0)$ would contain an element of $\overline{\mathbb{Q}}$ having all its conjugates in $\mathbb{E}$, but no such elements exist.  Hence 
 $\gamma_{\mathrm{F}}(\mathbb{E},\infty) = \gamma_{\mathrm{F}}^+(\mathbb{E},\infty) = 0$.  }
 \end{example}

\subsection{\bf RL-domains for the finite adeles}

We we will assume in this section the notation and hypotheses
of Theorem \ref{thm:urk}.  We identify $H^0(\X,O_{\X}(\X_1))$
with the $F$-subspace of the function field $F(\X)$ whose nonzero
elements $f$ have the property that $\mathrm{div}(f) + \X_1 \ge 0$.  We may assume that
the projective embedding $\X \to \mathbb{P}^t$
is assocated to a generating set $\{f_0,\ldots,f_t\}$ for  $H^0(\X,O_{\X}(\X_1))$
 such that $f_0 = 1$ and the $f_i$
for $i > 0 $ are nonconstant.  Our hypothesis is equivalent to the statement that
for $v \in M_{\mathrm{f}}(F)$ one has 
\begin{equation}
\label{eq:urkU}
U_v = \{x \in \X(\mathbb{C}_v): |f_i(x)|_v \le 1 \quad \mathrm{for \ all } \quad i\}.
\end{equation}

Let $O_F$ be the ring of integers of $F$ in the number field case and let $O_F$ be the 
ring of functions regular off of $M_\infty(F)$ in the function field case. 
Let $A = \bigoplus_{j = 0}^\infty A_j$ be the graded $O_F$-algebra
for which $A_0 = O_F$, $A_1$ is the $O_F$-submodule of $H^0(\X,O_{\X}(\X_1)) \subset F(\X)$
generated by $f_0,\ldots,f_t$ and $A_j$ is the $O_F$-module generated by all products
of $j$-elements of $A_1$ if $j \ge 2$.  We then have a projective flat model 
$\Xc$ of $\X$ over $O_F$ given by $\Xc = \mathrm{Proj}(A)$.  This
comes with a canonical very ample line bundle $O_\Xc(1)$ whch is
associated to the effective horizontal very ample Cartier divisor $D$
whose ideal sheaf is $O_\Xc(-1)$.  The zero locus of the degree one element $f_0$ of $A$ is
$D$.  This implies that the general fiber of $D$ is $\X_1$ and that $D$ 
is the Zariski closure of $\X_1$.  In particular, $D$ is horizontal.  The description (\ref{eq:urkU}) 
implies that for $v \in M_{\mathrm{f}}(F)$, the points $x \in \Xc(\mathbb{C}_v) = \X(\mathbb{C}_v)$
which lie in $U_v$ are exactly those for which the reduction of $x$ 
at $v$ does not lie on the reduction of $D$.  

We now apply Theorem \ref{thm:main} to conclude that there
is a finite morphism $\pi:\Xc \to \mathbb{P}^d_{O_F}$ such that
the pullback $\pi^{-1}(H)$ of the hyperplane $H = \{y _0 = 0\}$
in $\mathbb{P}^d_{O_F}$  is a positive multiple of $D$.  
The above description of $U_v$ via the reduction of $D$ at $v  \in M_{\mathrm{f}}(F)$
now implies $U_v = \pi^{-1}(B_v(1))$. Hence the restriction 
$\pi_{\mathbb{U}_{\mathrm{f}}}:\X \to \mathbb{P}^d_F$ of $\pi$ to the generic fiber
$\X$ of $\Xc$ has the properties stated in Theorem \ref{thm:urk}.
Thus if $U_v$ for $v \in M_\infty$ is a large enough open set
such that $\mathbb{U} = \mathbb{U}_{\mathrm{f}} \times \prod_{v \in M_\infty(F)} U_v$
is an open adelic set satisfying the standard hypotheses, there
will be an adelic polydisk $B(r)$ in $A^d$ such that $|r| > 1$
and $\pi^{-1}(B(r)) \subset \mathbb{U}$.  Therefore $\gamma_{\mathrm{F}}(\mathbb{U},\X_1) > 1$, 
which completes
the proofs of Theorem \ref{thm:urk} and Corollary \ref{cor:app}.

\end{document}